\documentclass{amsart}
\usepackage{verbatim}
\usepackage[textsize=scriptsize]{todonotes}
\usepackage{mathtools}
\usepackage{xcolor}
\usepackage{tikz-cd}
\usetikzlibrary{decorations.pathmorphing}

\usepackage{amsmath}
\usepackage{amssymb}
\usepackage{amsthm}
\usepackage{amscd}
\usepackage{enumerate}
\usepackage[pdfusetitle,unicode,hidelinks]{hyperref}
\usepackage{bbm}
\usepackage{etoolbox}

\usepackage[utf8]{inputenc}

\newtheorem{proposition}{Proposition}
\newtheorem{corollary}[proposition]{Corollary}
\newtheorem{lemma}[proposition]{Lemma}
\newtheorem{theorem}[proposition]{Theorem}

\newtheorem*{conjecture*}{Conjecture}
\newtheorem*{theorem*}{Theorem}
\newtheorem*{corollary*}{Corollary}
\newtheorem*{proposition*}{Proposition}
\newtheorem*{lemma*}{Lemma}
\theoremstyle{definition}
\newtheorem{definition}[proposition]{Definition}
\newtheorem{construction}[proposition]{Construction}

\newtheorem*{definition*}{Definition}
\newtheorem*{construction*}{Construction}
\theoremstyle{remark}
\newtheorem{remark}[proposition]{Remark}
\newtheorem{example}[proposition]{Example}
\newtheorem*{example*}{Example}

\newcommand{\id}{\operatorname{id}}
\newcommand{\Z}{\mathbb{Z}}

\let\scr=\mathcal
\let\bb=\mathbb
\newcommand{\Gm}{{\mathbb{G}_m}}
\newcommand{\Gmp}[1]{{\mathbb{G}_m^{\wedge #1}}}
\def\A{\bb A}
\def\P{\bb P}

\newcommand{\1}{\mathbbm{1}}

\newcommand{\eff}{{\text{eff}}}

\newcommand{\SH}{\mathcal{SH}}
 
\newcommand{\Spc}{\mathrm{Spc}}

\DeclareMathOperator*{\colim}{colim}

\let\lim=\relax
\DeclareMathOperator*{\lim}{lim}

\def\Spc{\mathcal{S}\mathrm{pc}{}}

\def\Fun{\mathrm{Fun}}

\newcommand{\wequi}{\simeq}

\def\adj{\leftrightarrows}

\DeclareRobustCommand{\ul}{\underline}

\newcommand{\heart}{\heartsuit}
\newcommand{\tr}{\mathrm{tr}}

\newcommand{\iHom}{\ul{\operatorname{Hom}}}

\newcommand{\HI}{\mathbf{HI}}

\let\cat=\mathrm
\def\Sm{{\cat{S}\mathrm{m}}}

\numberwithin{proposition}{section}

\usepackage{etoolbox}
\newtoggle{draft}

\iftoggle{draft} {
\usepackage[margin=1.5in]{geometry}

\newcommand{\NB}[1]{\todo[color=gray!40]{#1}}
\newcommand{\TODO}[1]{\todo[color=red]{#1}}
\newcommand{\tom}[1]{\todo[color=green]{#1}}
\newcommand{\mura}[1]{\todo[color=yellow]{#1}}
}{ 
\usepackage[margin=1in]{geometry}

\newcommand{\NB}[1]{}
\newcommand{\TODO}[1]{}
\newcommand{\tom}[1]{}
\newcommand{\mura}[1]{}
\renewcommand{\todo}[1]{}
}
\usepackage{hyperref}

\newcommand{\lra}[1]{\langle #1 \rangle}
\newcommand{\SHS}{\mathcal{SH}^{S^1}\!}

\newcommand{\ord}{\mathrm{ord}}

\def\ph{\mathord-}

\title{The zeroth $\P^1$-stable homotopy sheaf of a motivic space}
\date{\today}
\author{Tom Bachmann}
\email{tom.bachmann@zoho.com}

\begin{document}
\maketitle

\begin{abstract}
We establish a kind of ``degree zero Freudenthal $\Gm$-suspension theorem'' in motivic homotopy theory.
From this we deduce results about the conservativity of the $\P^1$-stabilization functor.

In order to establish these results, we show how to compute certain pullbacks in the cohomology of a strictly homotopy invariant sheaf in terms of the Rost--Schmid complex.
This establishes the main conjecture of \cite{bachmann-moving}, which easily implies the aforementioned results.
\end{abstract}

\tableofcontents

\section{Introduction}
After recalling some preliminaries in \S\ref{sec:preliminaries}, this article has two main sections of very different flavor.
In \S\ref{sec:formula} we establish a technical result about Rost--Schmid complexes of strictly homotopy invariant sheaves.
Then in \S\ref{sec:app} we draw applications to the stabilization problem in motivic homotopy theory.
We now describe these two main sections in reverse order, and then we sketch their relation.
For more background and motivation, the reader may wish to consult the introduction of \cite{bachmann-moving}.

\subsection{$\P^1$-stabilization in motivic homotopy theory}
Motivic homotopy theory is the universal homotopy theory of smooth algebraic varieties, say over a field $k$.
It is built by freely adjoining homotopy colimits to the category of smooth $k$-varieties, and then enforcing Nisnevich descent and making $\A^1$ contractible \cite{A1-homotopy-theory}.
Write $\Spc(k)_*$ for the pointed version of this theory.\footnote{We think of this as an $\infty$-category, but no information will be lost for the purposes of this introduction by just considering its homotopy $1$-category.}
This is a symmetric monoidal category (the monoidal operation being given by the smash product), and every pointed smooth variety defines an object in it.
Given a pointed motivic space $\scr X \in \Spc(k)_*$, the classical \emph{homotopy groups} upgrade to \emph{homotopy sheaves}\footnote{I.e. Nisnevich sheaves on the site of smooth $k$-varieties.} $\ul{\pi}_i(\scr X)$.

The Riemann sphere $\P^1 := (\P^1, 1) \in \Spc(k)_*$ plays a similar role to the ordinary sphere in classical topology.
\emph{Stable} motivic homotopy theory is concerned with the category obtained by making $\Sigma_{\P^1} := \wedge \P^1$ into an equivalence.
It is this context in which algebraic cycles and motivic cohomology naturally appear.
We can take a more pedestrian approach.
The functor $\Sigma_{\P^1}$ has a right adjoint $\Omega_{\P^1}$, and there is a directed diagram of endofunctors of $\Spc(k)_*$ \[ \id \to \Omega_{\P^1} \Sigma_{\P^1} =: Q_1 \to \Omega_{\P^1}^2 \Sigma_{\P^1}^2 =: Q_2 \to \dots \to \Omega_{\P^1}^n \Sigma_{\P^1}^n =: Q_n \to \dots; \] denote by $Q$ its homotopy colimit.
Then $Q\scr X$ is the \emph{$\P^1$-stabilization of $\scr X$}, and the homotopy sheaves of $Q\scr X$ are called the $\P^1$-stable homotopy sheaves of $\scr X$.

A simple form of our main application of our technical result is as follows.
It is reminiscent of the fact that for an ordinary space $X$, the sequence of sets $\{\pi_0 \Omega^i \Sigma^i X\}_{i \ge 0}$ is given by $\pi_0 X$, $F \pi_0 X$, $\Z(\pi_0 X)$, $\Z(\pi_0 X)$, \dots, where for a pointed set $A$, $FA$ denotes the free group on $A$ (with identity given by the base point), and $\Z(A)$ denotes the free abelian group on $A$ (with $0$ given by the base point).
\begin{theorem} \label{thm:intro-main}
Let $k$ be a perfect field and $n \ge 3$ (if $char(k) = 0$, $n=2$ is also allowed).
Then for $\scr X \in \Spc(k)_*$, the canonical map \[ \ul{\pi}_0 Q_n \scr X \to \ul{\pi}_0 Q \scr X \] is an isomorphism.
\end{theorem}
\begin{proof}
This is an immediate consequence of Corollary \ref{cor:freudenthal} and e.g. Morel's Hurewicz theorem \cite[Theorem 6.37]{A1-alg-top}.
\end{proof}

\begin{example}
Morel's computations \cite[Corollary 6.43]{A1-alg-top} imply that for $\scr X = S^0$, already $\ul{\pi}_0 Q_2 S^0 \wequi \ul{GW} \wequi \ul{\pi}_0 Q S^0$.
Our result shows that this stabilization is not special to $S^0$, except that our results are not strong enough to establish stabilization at $Q_2$, only at $Q_3$.
See also Remark \ref{rmk:speculation}.
\end{example}

We also obtain some conservativity results; here is a simple form.
It is similar to the fact that stabilization is conservative on simply connected topological spaces.
Write $\Spc(k)_*(n) \subset \Spc(k)_*$ for the subcategory generated under homotopy colimits by objects of the form $X_+ \wedge \Gmp{n}$ with $X \in \Sm_k$ (and $\Gm := (\A^1 \setminus 0, 1)$, $X_+ := X \coprod *$).
Denote by $\Spc(k)_{*,\ge 1} \subset \Spc(k)_*$ the subcategory of $\A^1$-simply connected spaces.
\begin{theorem}
Let $k$ be perfect, and put $n=1$ if $char(k) = 0$ and $n = 3$ if $char(k) > 0$.
Then the stabilization functor \[ Q: \Spc(k)_{*,\ge 1} \cap \Spc(k)_*(n) \to \Spc(k)_* \] is conservative (i.e., detects equivalences).
\end{theorem}
In particular $\Sigma_{\P^1}$ and all of its iterates, and also $\Sigma_{\P^1}^\infty$, are conservative on the same subcategory.
\begin{proof}
This is an immediate consequence of Corollary \ref{cor:conservative} and e.g. \cite[Corollary 2.23]{wickelgren2014simplicial}.
\end{proof}

The results in \S\ref{sec:app} are stronger than the sample given above; in fact they are stated in terms of the stabilization functor from $S^1$-spectra to $\P^1$-spectra.
The reader is encouraged to skip to this section directly.
Our main results in the form of Corollary \ref{cor:freudenthal}, Theorem \ref{thm:convergence} and Corollary \ref{cor:conservative} can be understood without reading the rest of the article (except perhaps for taking a glance at \S\ref{subsec:app-notation}, where some notation is introduced).

\subsection{Pullbacks and the Rost--Schmid complex}
The results sketched above are obtained by combining the main results of \cite{bachmann-moving} with a technical result that we describe now.
Essentially, this establishes \cite[Conjecture 6.10]{bachmann-moving} (for $n \ge 3$); all our applications are a consequence of this and were already anticipated when writing \cite{bachmann-moving}.

Let $M$ be a strictly homotopy invariant sheaf (see \S\ref{sec:preliminaries} for this and related notions, and a more complete account of the following sketch) and $X$ a smooth variety.
Morel has proved \cite[Corollary 5.43]{A1-alg-top} that there is a very convenient complex, known as the \emph{Rost--Schmid complex} $C^*(X, M)$, which can be used to compute the Nisnevich cohomology $H^*(X, M)$.
This complex has the special property that $C^n(X, M)$ only depends on the $n$-fold contraction $M_{-n}$, and similarly so does the boundary map $C^n(X, M) \to C^{n+1}(X, M)$.
Let $Z \subset X$ have codimension $\ge d$.
An obvious modification $C^*_Z(X, M)$ of $C^*(X,M)$ can be used to compute $H^*_Z(X, M)$; by construction one has $C^n_Z(X, M) = 0$ for $n < d$.
It follows that the group $H^d_Z(X, M)$ only depends on $M_{-d}$ (in fact this holds for all groups $H^*_Z(X, M)$, but we are most interested in the lowest one).
Now let $f: Y \to X$ be a morphism of smooth varieties with $f^{-1}(Z)$ also of codimension $\ge d$ on $Y$.
Then the pullback map \begin{equation} \label{eq:intro} f^*: H^d_Z(X, M) \to H^d_{f^{-1}(Z)}(Y, M) \end{equation} is a morphism of abelian groups, both of which only depend on $M_{-d}$.

It is not difficult to show (using the results of \cite{bachmann-moving}; see the proof of Theorem \ref{thm:hearts} for details) that \cite[Conjecture 6.10]{bachmann-moving} is equivalent to the statement that the morphism \eqref{eq:intro} also only depends on $M_{-d}$, in an appropriate sense.\footnote{In fact, our original plan for \cite{bachmann-moving} was to establish \cite[Conjecture 6.10]{bachmann-moving} (and hence the results in \S\ref{sec:app}) by proving that $f^*$ only depends on $M_{-d}$. This turned out to be more difficult than we had anticipated.}
The main result of this article (Theorem \ref{thm:formula}) states that this is true.

We establish this by adapting an argument of Levine, using a variant of Gabber's presentation lemma to set up an induction on $d$. (The case $d=0$ holds tautologically.)

\subsection{From pullbacks to stabilization}
This article brings to conclusion a program started in \cite{bachmann-moving}.
There we developed the following strategy for establishing stabilization results such as Theorem \ref{thm:intro-main}.
First we note that $S^1$-stabilization is well-understood and behaves largely as in topology; thus it suffices to prove the analogous result for motivic $S^1$-spectra.
(For detailed definitions of this and the following notions, see \S\ref{subsec:app-notation}.)
Write $\SHS(k)(d) \subset \SHS(k)$ for the localizing subcategory generated by $d$-fold $\Gm$-suspensions, and similarly $\SH(k)^\eff(d) \subset \SH(k)$ for the localizing subcategory generated by the image of $\SHS(k)(d)$.
These categories afford $t$-structures induced by the canonical generating sets, and hence the stabilization functor $\SHS(k)(d) \to \SH(k)^\eff(d)$ is right-$t$-exact.
One finds that in order to prove stabilization results, it will be enough to show that the induced functor on hearts $\SHS(k)(d)^\heart \to \SH(k)^\eff(d)^\heart$ is an equivalence.
Since the right hand category is by now well-understood, let us focus on the left hand side.
It is not difficult to show that the functor of $d$-fold $\Gm$-loops $\SHS(k)(d)^\heart \to \SHS(k)^\heart \wequi \HI(k)$ is monadic.
In other words, we may think of objects of $\SHS(k)(d)^\heart$ as strictly homotopy invariant sheaves with extra structure.
One way of phrasing the main result of \cite{bachmann-moving} is (see e.g. \cite[Remark 4.17]{bachmann-moving}) that this extra structure is precisely the data of closed pullbacks on cohomology with support in codimension $d$.
These are precisely the kinds of maps that we show only depend on $M_{-d}$ ``in an appropriate sense''.
To be more specific, the appropriate sense is that $M_{-d}$ is a so-called sheaf with $\A^1$-transfers (see Remark \ref{rmk:A1-transfers} for details), and the pullback only depends on this additional structure.

All of this more or less implies\footnote{Making these arguments precise requires some further effort; for this reason in \S\ref{subsec:app-heart} we follow a slightly different strategy.} that $\SHS(k)(d)^\heart$ is equivalent to the full subcategory of the category of sheaves with $\A^1$-transfers on objects of the form $M_{-d}$.
It follows from \cite[Theorem 5.19]{bachmann-moving} that for $d$ big enough, this subcategory is equivalent to $\SH(k)^\eff(d)^\heart$, as desired.

\subsection{Acknowledgements}
It is my pleasure to thank Maria Yakerson, Marc Levine and Mike Hopkins for fruitful discussions about these problems.

I would also like to thank\NB{keep alphabetical} Marc Hoyois and Maria Yakerson for comments on a draft of this article.

\subsection{Notation and conventions}
We fix throughout a field $k$.
All non-trivial results will require $k$ to be perfect.

Given a presheaf $M$ on the category of smooth varieties over $k$, and an essentially smooth $k$-scheme $X$, we denote by $M(X)$ the evaluation at $X$ of the canonical extension of $M$ to pro-(smooth schemes), into which the category of essentially smooth schemes embeds by \cite[Proposition 8.13.5]{EGAIV}.
In other words, if $X = \lim_i X_i$ is a cofiltered limit of smooth $k$-schemes with affine transition maps, then $M(X) = \colim_i M(X_i)$ (and this is known to be independent of the presentation of $X$).

Given a scheme $X$ and a point $x \in X$, we identify $x$ and $Spec(k(x))$.
In particular, if $X$ is smooth and $k$ is perfect (so that $x$ is essentially smooth), then we write $M(x)$ for what is often denoted $M(k(x))$.

Given a scheme $X$ and $d \ge 0$, we write $X^{(d)}$ for the set of points of $X$ of codimension $d$ on $X$; in other words if $x \in X$ then $x \in X^{(d)}$ if and only if $\dim X_x = d$, where $X_x$ denotes the localization of $X$ in $x$.
For example, $X^{(0)}$ is the set of generic points of $X$.

For a regular immersion $Y \hookrightarrow X$, we denote by $N_{Y/X}$ the normal bundle and by $\omega_{Y/X} = \det N_{Y/X}^{\vee}$ the determinant of the conormal bundle.
More generally for any morphism $Y \to X$ such that the cotangent complex $L_{Y/X}$ is perfect we write $\omega_{Y/X} = \det L_{Y/X}$.

\section{Preliminaries} \label{sec:preliminaries}
We recall some well-known results from motivic homotopy theory.

\subsection{Strictly homotopy invariant sheaves}
We write $\Sm_k$ for the category of smooth $k$-schemes.
We make it into a site by endowing it with the Nisnevich topology \cite{nisnevich1989completely}.
This is the only topology we shall use; all cohomology will be with respect to it.
Unless noted otherwise, by a (pre)sheaf we mean a (pre)sheaf of abelian groups on $\Sm_k$.

Recall that a sheaf $M$ is called \emph{strictly homotopy invariant} if, for all $X \in \Sm_k$, the canonical map $H^*(X, M) \to H^*(\A^1 \times X, M)$ is an isomorphism.
We denote the category of strictly homotopy invariant sheaves by $\HI(k)$.

\begin{example}
For a commutative ring $A$, denote by $GW(A)$ its Grothendieck-Witt ring, i.e. the additive group completion of the semiring of isometry classes of non-degenerate, symmetric bilinear forms on $A$ \cite{milnor1973symmetric}.
Write $\ul{GW}$ for the associated Nisnevich sheaf on $\Sm_k$.
Then $\ul{GW}$ turns out to be strictly homotopy invariant (combine \cite[Theorem A]{panin-witt-purity} and \cite[\S\S 2,3]{A1-alg-top}).
\end{example}

\begin{remark} \label{rmk:pullback-defn}
As mentioned in the introduction, there exists a universal homotopy theory built out of (pointed) smooth varieties by enforcing $\A^1$-homotopy invariance and Nisnevich descent \cite{A1-homotopy-theory}; we denote it by $\Spc(k)_*$.
By construction, for $M \in \HI(k)$, the Eilenberg-MacLane spaces $K(M, i)$ define objects in $\Spc(k)_*$.
In this way, results about $\Spc(k)_*$ translate into properties of the cohomology of strictly homotopy invariant sheaves.
For example given $X \in \Sm_k$ and $Z \subset X$ closed we have an isomorphism \[ H^i_Z(X,M) \wequi [X/X \setminus Z, K(M,i)]_{\Spc(k)_*}, \] and given a morphism $f: Y \to X$ (respectively a closed subset $Z' \subset Z$) we have morphisms $Y/Y \setminus f^{-1}(Z) \to X/X \setminus Z$ (respectively $X/X \setminus Z \to X/X \setminus Z'$) inducing the pullback $H^i_Z(X, M) \to H^i_{f^{-1}(Z)}(Y, M)$ (respectively the extension of support map $H^i_{Z'}(X, M) \to H^i_Z(X,M)$).
We will use this correspondence freely in the sequel.
\end{remark}

\subsubsection{Unramifiedness}
Let $X \in \Sm_k$ be connected and $\emptyset \ne U \subset X$ be open.
Then for $M \in \HI(k)$, the canonical map $M(X) \to M(U)$ is an injection \cite[Lemma 6.4.4]{morel2005stable}.
It follows that if $\xi \in X$ is the generic point, then $M(X) \hookrightarrow M(\xi)$.

\subsubsection{Contractions}
For a presheaf $M$, write $M_{-1}$ for the presheaf given by $M_{-1}(X) = ker(M((\A^1 \setminus 0) \times X) \xrightarrow{i_1^*} M(X))$ and $M_{-n}$ for the $n$-fold iteration of this construction.
Here $i_1: X \to (\A^1 \setminus 0) \times X$ denotes the inclusion at $1 \in \A^1$.
Pullback along the structure map splits $i_1^*$ and hence $M_{-1}$ is a summand of the internal mapping object $\iHom(\A^1 \setminus 0, M)$.
It follows that $M_{-n}$ is a ((strictly\NB{this needs an extra argument...}) homotopy invariant) (pre)sheaf if $M$ is.

\begin{example}
We have $H^1(\P^1_K, M) \wequi M_{-1}(K)$, for any finitely generated separable field extension $K/k$.
Indeed we can cover $\P^1_K$ by two copies of $\A^1_K$ with intersection $(\A^1 \setminus 0)_K$ and $H^i(\A^1_K, M) \wequi H^i(Spec(K), M)$ (whence in particular $H^1(\A^1_K, M) = 0$), so the claim follows from the Mayer--Vietoris sequence for this covering.
\end{example}

\subsubsection{$GW$-module structure} \label{subsec:GW-module}
Let $X \in \Sm_k$ and $u \in \scr O(X)^\times$.
Multiplication by $u$ defines an endomorphism of $(\A^1 \setminus 0) \times X$ and hence of $\iHom(\A^1 \setminus 0, M)(X)$; passing to the summand we obtain $\lra{u}: M_{-1}(X) \to M_{-1}(X)$.
Suppose that $M \in \HI(k)$.
Since the map $\Z[\scr O^\times] \to \ul{GW}, u \mapsto \lra{u}$ is surjective on fields, unramifiedness implies that this construction extends in at most one way to a $\ul{GW}$-module structure on $M_{-1}$.
It turns out that this $\ul{GW}$-module structure always exists \cite[Lemma 3.49]{A1-alg-top}.

\subsubsection{Twisting}
Given a line bundle $\scr L$ on $X \in \Sm_k$, write $\scr L^\times$ for the sheaf of non-vanishing sections.
For $M \in \HI(k)$ and $d > 0$ we put $M_{-d}(X, \scr L) = H^0(X, M_{-d} \times_{\scr O^\times} \scr L^\times)$; here the action of $\scr O^\times$ on $M_{-d}$ is via $\scr O^\times \to \ul{GW}$ and the action on $\scr L^\times$ is given by multiplication.
Note that since $\lra{u} = \lra{u^{-1}}$ we have $M_{-d}(X, \scr L) \wequi M_{-d}(X, \scr L^{-1})$.

\subsubsection{Thom spaces} \label{subsec:thom}
For $X \in \Sm_k$ and $V$ a vector bundle on $X$ of rank $d$, we have $Th(V) := V/V \setminus 0_X \in \Spc(k)_*$.
For $M \in \HI(k)$, there are canonical isomorphisms \cite[Lemma 5.35]{A1-alg-top} \[ [V/V \setminus 0_X, K(M, d)]_{\Spc(k)_*} \wequi H^d_{0_X}(V, M) \wequi M_{-d}(X, \det V). \]

\subsubsection{Homotopy purity} \label{subsec:htpy-purity}
Let $X \in \Sm_k$, $U \subset X$ open with reduced closed complement $Z = X \setminus U$ also smooth.
Then in $\Spc(k)_*$ there is a canonical equivalence \cite[\S3 Theorem 2.23]{A1-homotopy-theory} \[ X/X \setminus Z \wequi Th(N_{Z/X}). \]

\subsubsection{Boundary maps} \label{subsec:boundary}
Let $X \in \Sm_k$ and $x \in X^{(d)}$.
Then $X_x$ is an essentially smooth scheme with closed point $x$.
Homotopy purity supplies us with the collapse sequence\footnote{Here and in the sequel we view essentially smooth schemes as defining pro-objects in $\Spc(k)$.} \[ X_x \to X_x/X_x \setminus x \wequi Th(N_{x/X}) \xrightarrow{\partial} \Sigma(X_x \setminus x). \]
Pullback along $\partial$ induces the boundary map in the long exact sequence of cohomology with support.
We most commonly use the case where $d=1$.
Then $X_x \setminus x = \xi$ where $\xi$ is the generic point of $X$ (specializing to $x$), and the boundary map takes the familiar form \[ \partial: M(\xi) \to M_{-1}(x, \omega_{x/X}). \]

\subsubsection{Monogeneic transfers} \label{subsec:transfer}
Let $k$ be perfect and $K/k$ be a finitely generated field extension, whence $X = Spec(K)$ is an essentially smooth scheme.
Let $K(x)/K$ be a finite, monogeneic field extension.
We are supplied with an embedding $X' = Spec(K(x)) \xhookrightarrow{x} \A^1_X \subset \P^1_X$ and thus homotopy purity provides us with a collapse map \[ \P^1_X \to \P^1_X/\P^1_X \setminus X' \wequi Th(N_{X'/\A^1_X}) \wequi Th(\scr O_{X'}); \] here the normal bundle is canonically trivialized by the minimal polynomial of $x$.
Pullback along this collapse map induces the \emph{monogeneic transfer}\footnote{Morel calls this the \emph{geometric transfer}.} \cite[p.99]{A1-alg-top} \[ \tau_x: M_{-1}(K(x)) \to M_{-1}(K). \]
Slightly more generally, suppose that $z \in \P^1_K$ is any closed point.
Then we have the transfer map \[ \tr_z: M_{-1}(z, \omega_{z/\P^1_K}) \wequi H^1_z(\P^1_K, M) \to H^1(\P^1_K, M) \wequi M_{-1}(K). \]
This contains no new information: if $z \in \A^1_K$ then $\tr_z$ coincides up to isomorphism with $\tau_z$, and the only other case is $z = \infty$ which is a rational point, and so $\tr_z$ is isomorphic to the identity.

\subsection{Cousin and Rost--Schmid complexes}
Let $M$ be a sheaf of abelian groups on $X$.
The cohomology of $M$ on $X$ can be computed using the \emph{coniveau spectral sequence}; see e.g. \cite[\S1]{bloch-ogus-gabber}.
On the zero line of the $E_1$ page one finds the so-called \emph{Cousin complex} \cite[(1.3)]{bloch-ogus-gabber} \begin{equation}\label{eq:cousin} M(X) \to \bigoplus_{x \in X^{(0)}} M(x) \to \bigoplus_{x \in X^{(1)}} H^1_x(X, M) \to \dots \to \bigoplus_{x \in X^{(d)}} H^d_x(X, M) =: C^d(X, M) \to \cdots. \end{equation}
Here \begin{equation} \label{eq:cousin-terms} H^d_x(X, M) = \colim_{V \ni x} H^d_{\overline{\{x\}} \cap V}(V, X), \end{equation} where the colimit runs over open neighborhoods of $x$.
The boundary maps in \eqref{eq:cousin} are induced by certain boundary maps in long exact sequences of cohomology with support.

Now suppose that $M \in \HI(k)$.
By the Bloch--Ogus--Gabber theorem \cite[Theorem 6.2.1]{bloch-ogus-gabber}, the Cousin complex \eqref{eq:cousin} is then \emph{exact} when viewed as a complex of sheaves (i.e. for $X$ local).
Since it consists of flasque sheafes, it can thus be used to compute the Zariski cohomology of $M$.
The terms also turn out to be Nisnevich-acyclic (see \cite[Theorem 8.3.1]{bloch-ogus-gabber} or \cite[Lemma 5.42]{A1-alg-top}), and hence the Cousin complex computes the Nisnevich cohomology of $M$ as well (which thus turns out to coincide with the Zariski cohomology).

\begin{remark} \label{rmk:cousin-support}
The Cousin complex can also be used to compute cohomology with support in a closed subscheme $Z$; just replace \[ C^d(X, M) = \bigoplus_{x \in X^{(d)}} H^d_x(X, M) \quad\text{by}\quad C^d_Z(X, M) = \bigoplus_{x \in X^{(d)} \cap Z} H^d_x(X, M). \]
This holds since the resolving sheaves are flasque.
\end{remark}

Now let $k$ be perfect.
We would like to make this complex more explicit.
As a first step, homotopy purity (see \S\S\ref{subsec:thom},\ref{subsec:htpy-purity}) allows us to identify the groups \eqref{eq:cousin-terms} in the Cousin complex more explicitly as \[ H^d_x(X, M) \wequi M_{-d}(x, \omega_{x/X}). \]
Indeed by generic smoothness, shrinking $V$ if necessary we may assume that $Z := \overline{\{x\}} \cap V$ is smooth of codimension $d$ on $V$, and then \[ H^d_Z(V, M) \wequi [V/V \setminus Z, K(d,M)] \wequi [Th(N_{Z/V}), K(d,M)] \wequi M_{-d}(Z, \omega_{Z/V}); \] the claim now follows by taking colimits.
The boundary maps of the Cousin complex can also be identified.
We only use the following weak form of this result.

\begin{theorem}[Morel \cite{A1-alg-top}, Corollary 5.44] \label{thm:morel}
Let $M \in \HI(k)$, $X \in \Sm_k$ and $d \ge 0$.
Then the boundary map $C^d(X, M) \to C^{d+1}(X, M)$ in the Cousin complex only depends on the sheaf $M_{-d}$, together with (if $d>0$) its structure as a $\ul{GW}$-module from \S\ref{subsec:GW-module}.
\end{theorem}

In fact, Morel proves this result by identifying the Cousin complex with another complex called the \emph{Rost--Schmid} complex (which has the same terms but a priori different boundary maps).
In other words the boundary map in the Cousin complex admits an explicit formula, involving only the codimension $1$ boundary of \S\ref{subsec:boundary} and (composites of) the monogeneic transfer of \S\ref{subsec:transfer}.\footnote{The monogeneic transfer on $M_{-d}$ is extra structure, determined by the presentation of $M_{-d}$ as a contraction. But the boundary $C^d \to C^{d+1}$ involves the monogeneic transfer only on $M_{-d-1}$, and so indeed only depends on the sheaf $M_{-d}$.}\NB{But even the weak form seems highly non-trivial already.}

In the sequel, we will not distinguish between the Cousin and Rost--Schmid complexes.

\section{A ``formula'' for closed pullback} \label{sec:formula}
In this section we establish our main result.
\begin{theorem} \label{thm:formula}
Let $k$ be a perfect field, $M \in \HI(k)$, $f: Y \to X \in \Sm_k$, $d \ge 1$, $Z \subset X$ closed of codimension $\ge d$ such that $f^{-1}(Z) \subset Y$ is also of codimension $\ge d$.
Then the map \[ f^*: H^d_Z(X, M) \to H^d_{f^{-1}(Z)}(Y, M) \] only depends on $M_{-d} \in \HI(k)$, together with its $GW$-module structure and transfers along monogeneic field extensions (in the sense of \S\S\ref{subsec:GW-module},\ref{subsec:transfer}).
\end{theorem}
In the sequel, we shall say ``depends only on $M_{-d}$'' to mean what is asserted in the theorem, i.e. ``depends only on $M_{-d}$ as a $GW$-module with transfers''.
\begin{remark} \label{rmk:A1-transfers}
Let us make precise the notion that ``$f^*$ only depends on $M_{-d}$''.

For this first recall from \cite[\S5.1]{bachmann-moving} the notion of a presheaf with $\A^1$-transfers.
This is just a presheaf $F$ on $\Sm_k$ together with for every finitely generated field $K/k$ a $GW(K)$-module structure on $F(K)$, and for every finite monogeneic extension $K(x)/K$ a transfer $\tau_x: F(K(x)) \to F(K)$.
A morphism of presheaves with $\A^1$-transfers is a morphisms of sheaves which commutes with the $GW(K)$-module structures and the transfers.
If $M \in \HI(k)$ and $d \ge 1$, then $M_{-d}$ acquires the structure of a presheaf with $\A^1$-transfers (see \S\S\ref{subsec:GW-module},\ref{subsec:transfer} or \cite[Example 5.2]{bachmann-moving}).

Now suppose given $M, N \in \HI(k)$ and an isomorphism $M_{-d} \wequi N_{-d}$ of presheaves with $\A^1$-transfers.
Then there is a canonical induced isomorphism $H^d_Z(X, M) \wequi H^d_Z(X, N)$ (and similarly for $Y$), by identifying the Rost--Schmid resolutions with support in $Z$ (i.e. using Remark \ref{rmk:cousin-support} and Theorem \ref{thm:morel}; this does not even depend on the identification of the transfers).
The theorem asserts that this isomorphism is compatible with the pullback $f^*$ (and for this we crucially need the compatibility of the transfers).
\end{remark}

\begin{remark} \label{rmk:RS-basic}
Note that the Rost--Schmid complex is functorial in smooth morphisms\NB{in fact flat morphisms}  in an obvious way, so that the theorem is clear e.g. for $f$ an open immersion.
We will often use this in conjunction with the observation (which follows e.g. from the form of the Rost--Schmid complex) that if $Z$ has codimension $\ge d$ on $X$, then \[ H^d_Z(X, M) \hookrightarrow \bigoplus_z H^d_z(X_z, M), \] where the sum is over the (finitely many) generic points of $Z$ of codimension $d$ on $X$.
\end{remark}

If the support is smooth and the intersection is transverse, all is well.
\begin{lemma} \label{lemm:transverse-pullback}
Suppose that both $Z$ and $f^{-1}(Z)$ (with its induced scheme structure as a pullback) are smooth.
Then $f^*: H^d_Z(X, M) \to H^d_{f^{-1}(Z)}(Y, M)$ only depends on $M_{-d}$.
\end{lemma}
\begin{proof}
Since $Z$ is smooth (and so is $X$), we may write $Z = Z_0 \coprod Z_1$ where all components of $Z_0$ have codimension precisely $d$ on $X$, and all components of $Z_1$ have codimension $>d$.
Consider the commutative diagram
\begin{equation*}
\begin{CD}
H^d_Z(X, M) @>{f^*}>> H^d_{f^{-1}(Z)}(Y, M) \\
@AAA                           @AAA \\
H^d_{Z_0}(X, M) @>{f^*}>> H^d_{f^{-1}(Z_0)}(Y, M).
\end{CD}
\end{equation*}
Here the vertical maps are extension of support, and hence only depend on $M_{-d}$.
Moreover by construction the left hand vertical map is an isomorphism.
We may thus replace $Z$ by $Z_0$, i.e. assume that all components of $Z$ have codimension precisely $d$ on $X$.

Recall from Remark \ref{rmk:pullback-defn} that the pullback $f^*: H^d_Z(X, M) \to H^d_{f^{-1}(Z)}(Y, M)$ is induced by pullback along the map $Y/Y \setminus f^{-1}(Z) \to X/X \setminus Z \in \Spc(k)_*$ on $K(d,M)$.
Let $\eta$ be a generic point of $f^{-1}(Z)$ (necessarily of codimension $d$ on $Y$\NB{locally $Z$ is cut out by $d$ equations, so codimension $\le d$}).
Shrinking $X$ around $f(\eta)$ using Remark \ref{rmk:RS-basic}, we may assume that the normal bundle $N_{Z/X}$ is trivial.
Since $f: (Y, f^{-1}(Z)) \to (X, Z)$ is a morphism of smooth closed pairs \cite[\S3.5]{hoyois-equivariant}, the map $Y/Y \setminus f^{-1}(Z) \to X/X \setminus Z$ is equivalent to $Th(g)$, where $g: N_{f^{-1}(Z)/Y} \to N_{Z/X}$ is the map induced by $f$ \cite[Theorem 3.23]{hoyois-equivariant}.
Our assumptions on codimension imply that $f^* N_{Z/X} \wequi N_{f^{-1}(Z)/Y}$\NB{$Z$ is locally cut out by $d$ equations, which generated both normal bundles}, whence $g \wequi f|_{f^{-1}(Z)} \times \id_{\A^d}$.
Pullback along $Th(g) \wequi (f|_{f^{-1}(Z)})_+ \wedge T^d$ thus only depends on $M_{-d}$, and the result follows.
\end{proof}

Recall that for (sets, say, and hence presheaves of sets) $A \subset X, B \subset Y$, we have a canonical isomorphism \begin{equation}\label{eq:smash-formula} X/A \wedge Y/B \wequi X \times Y/(A \times Y \cup X \times B). \end{equation}
\begin{construction}
Let $Z \subset X \times \P^1$ be closed with image $Z'$ in $X$.\NB{We only ever apply this with $Z \subset \A^1_X$...}
Applying the isomorphism \eqref{eq:smash-formula} with $Y = \P^1$, $A=X \setminus Z'$, $B = \emptyset$ we obtain the following equivalence \[ X \times \P^1/(X \times \P^1 \setminus Z' \times \P^1) \wequi (X/X \setminus Z') \wedge \P^1_+. \]
Together with the stable splitting $\P^1_+ \wequi \1 \vee \P^1$ and extension of support this induces a map \[ \tr_Z: H^{d}_{Z}(X \times \P^1, M) \to H^{d}_{Z' \times \P^1}(X \times \P^1, M) \to H^{d-1}_{Z'}(X, M_{-1}). \]
\end{construction}
\begin{remark}\label{rmk:construction-functorial}
This construction is clearly functorial in $X$.
\end{remark}

\begin{lemma} \label{lemm:construction-is-transfer}
In the above notation, suppose that $Z \subset X \times \P^1$ has codimension $\ge d$ (so that $Z' \subset X$ has codimension $\ge d-1$).
Then $\tr_Z$ only depends on $M_{-d}$.
\end{lemma}
\begin{proof}
The transfer is given by pullback along the collapse map $\P^1_X/\P^1_X \setminus \P^1_{Z'} \to \P^1_X/\P^1_X \setminus Z$.
Remarks \ref{rmk:construction-functorial} and \ref{rmk:RS-basic} imply that the problem is local on $X$ around generic points of $Z'$ of codimension $d-1$; we may thus assume that $Z'$ is smooth \cite[Tag 0B8X]{stacks} and $N_{Z'/X}$ is trivial.
Lemma \ref{lemm:collapse} below identifies the transfer with the collapse map \[ N_{\P^1_{Z'}/\P^1_X}/N_{\P^1_{Z'}/\P^1_X} \setminus \P^1_{Z'} \to N_{\P^1_{Z'}/\P^1_X}/N_{\P^1_{Z'}/\P^1_X} \setminus Z \]
By assumption $N_{\P^1_{Z'}/\P^1_X}$ is trivial of rank $d-1$, so this map identifies with \[ \A^{d-1}_{\P^1_{Z'}} / \A^{d-1}_{\P^1_{Z'}} \setminus \P^1_{Z'} \to \A^{d-1}_{\P^1_{Z'}} / \A^{d-1}_{\P^1_{Z'}} \setminus Z. \]
Applying isomorphism \eqref{eq:smash-formula} with $X = \A^{d-1}$, $Y = \P^1_{Z'}$, $A = \A^{d-1} \setminus 0$ and respectively $B=\emptyset$ or $B=\P^1_{Z'} \setminus Z$, this identifies with \[ \A^{d-1}/\A^{d-1} \setminus 0 \wedge \P^1_{Z'+} \xrightarrow{\id \wedge t} \A^{d-1}/\A^{d-1} \setminus 0 \wedge \P^1_{Z'}/\P^1_{Z'} \setminus Z. \] (Use that $(\A^{d-1} \setminus 0) \times \P^1_{Z'} = \A^{d-1}_{\P^1_{Z'}} \setminus \P^1_{Z'}$ and $(\A^{d-1} \setminus 0) \times \P^1_{Z'} \cup \A^{d-1} \times (\P^1_{Z'} \setminus Z) = \A^{d-1}_{\P^1_{Z'}} \setminus Z$.)
Pullback along $t$ is the monogeneic transfer for $Z/Z'$, essentially by definition (see \S\ref{subsec:transfer}).
The result follows.
\end{proof}
\begin{remark} \label{rmk:construction-is-transfer}
The above proof shows that, on the level of the Rost--Schmid complex, the map $\tr_Z$ is given as follows.
For $z \in Z$ of codimension $e$ in $X \times \P^1$ and with image $z'$ of codimension $e-1$ in $X$, the map is given in components by \[ C^e(X \times \P^1, M) \supset M_{-e}(z, \omega_{z/X \times \P^1}) \wequi M_{-e}(z, \omega_{z/\P^1_{z'}} \otimes \omega_{\P^1_{z'}/\P^1_X}) \xrightarrow{\tr} M_{-e}(z', \omega_{z'/X}) \subset C^{e-1}(X, M_{-1}). \]
Here $\tr$ is the monogeneic transfer coming from the embedding $z \in \P^1_{z'}$.\NB{And if $z'$ has codimension $\ne e$, then the component is zero.}
\end{remark}

We used above the following form of the homotopy purity equivalence.
\begin{lemma} \label{lemm:collapse}
Let $Z \subset Y \subset X$ be closed immersions with $X, Y$ smooth.
Then the collapse map \[ X/(X \setminus Y) \to X/(X \setminus Z) \] is canonically homotopic to the collapse map \[ N_{Y/X}/N_{Y/X} \setminus Y \to N_{Y/X}/N_{Y/X} \setminus Z. \]
\end{lemma}
\begin{proof}
Write $X' = X \setminus Z$ and $Y' = Y \setminus Z$.
We can write the collapse map as \[ \frac{X}{X \setminus Y} \to \frac{X/X \setminus Y}{X'/X' \setminus Y'}. \]
Since $(X', Y') \to (X, Y)$ is a morphism of smooth closed pairs, it is compatible with purity equivalences \cite[after proof of Theorem 3.23]{hoyois-equivariant}, and so the collapse map identifies with \[ \frac{N_{Y/X}}{N_{Y/X} \setminus Y} \to \frac{N_{Y/X}/(N_{Y/X} \setminus Y)}{N_{Y'/X'}/(N_{Y'/X'} \setminus Y')} \wequi \frac{N_{Y/X}}{(N_{Y/X} \setminus Y) \cup N_{Y'/X'}} = \frac{N_{Y/X}}{N_{Y/X} \setminus Z}; \] see also \cite[top of p. 24]{hoyois-equivariant}.
This is the desired result.
\end{proof}

\begin{lemma} \label{lemm:pullback-transfer}
Let $X$ be (essentially) smooth, $i: Y \hookrightarrow X$ closed of codimension $1$ with $Y$ essentially smooth, $Z \subset X \times \P^1$ of codimension $\ge d$ such that $W := (Y \times \P^1) \cap Z$ also has codimension $\ge d$ on $Y \times \P^1$.
Write $Z', W'$ for the images of $Z$ and $W$ in $X, Y$, respectively.
Let $\eta_1, \dots, \eta_r$ be the generic points of $W$ of codimension $d$.
Suppose further that $Z \to Z'$ is quasi-finite\NB{i.e. finite} and $W \to W'$ is birational at $\eta_1$.
Then \[ i^*: H^d_Z(X \times \P^1, M) \to H^d_{W}(Y \times \P^1, M) \] only depends on $M_{-d}$, the map \[ i^*: H^{d-1}_{Z'}(X, M_{-1}) \to H^{d-1}_{W'}(Y, M_{-1}) \] and the maps \[ i^*_{\eta_j}: H^d_Z((X \times \P^1)_{\eta_j}, M) \to H^d_{W}((Y \times \P^1)_{\eta_j}, M), \] for $j > 1$.
\end{lemma}
In particular, the map $i^*$ does \emph{not} depend on $i_1^*$.
\begin{proof}
By Remarks \ref{rmk:construction-functorial} and \ref{rmk:construction-is-transfer}, we have a commutative diagram
\begin{equation*}
\begin{tikzcd}
H^d_Z(X \times \P^1, M) \ar[r, "i^*"] \ar[d, "\tr_Z"] & H^d_{W}(Y \times \P^1, M) \ar[d, "\tr_W"] \ar[r, hookrightarrow] & \bigoplus_i C^d_{\eta_i}(Y \times \P^1, M) \ar[d, "\tr_W"] \\
H^{d-1}_{Z'}(X, M_{-1}) \ar[r, "i^*"]                 & H^{d-1}_{W'}(Y, M_{-1}) \ar[r, hookrightarrow] & \bigoplus_{\eta \in W' \cap Y^{(d-1)}} C^{d-1}_\eta(Y, M_{-1}).
\end{tikzcd}
\end{equation*}
By Lemma \ref{lemm:construction-is-transfer}, the vertical maps only depend on $M_{-d}$, and it follows from Remark \ref{rmk:construction-is-transfer} and our assumption that $W \to W'$ is birational at $\eta_1$ that the right hand vertical map is injective on the component corresponding to $\eta_1$.
Let $a \in H^d_Z(X \times \P^1, M)$.
Write $i^*(a) = b_1 + \dots + b_r$, where $b_i \in C^d_{\eta_i}(Y \times \P^1, M)$.
For $j>1$, we know $i_{\eta_j}^*$, hence we know $b_j$ and thus we know $\tr_W(b_j)$.
Since we know the bottom horizontal map, we know $i^* \tr_Z(a) = \tr_W(i^*(a))$.
Consequently we know $\tr_W(b_1) = \tr_W(i^*a) - \sum_{j>1} \tr_W(b_j)$, and hence $b_1$.
This concludes the proof.
\end{proof}
\begin{example} \label{ex:induction-start}
If $d=1$, then the map $i^*: H^{d-1}_{Z'}(X, M_{-1}) \subset M_{-1}(X) \to H^{d-1}_{W'}(Y, M_{-1}) \subset M_{-1}(Y)$ clearly only depends on $M_{-1}$, as desired.
\end{example}
\begin{example}
If $Z$ is smooth and transverse to $Y$ at $\eta_j$ for $j>1$, then $i^*_{\eta_j}$ only depends on $M_{-d}$ by Lemma \ref{lemm:transverse-pullback}, as desired.
\end{example}

The following is the key reduction.
It is an adaptation of \cite[Lemma 7.2]{levine2010slices}.
\begin{lemma} \label{lemm:key}
Let $X, Y$ be (essentially) smooth, $i: Y \hookrightarrow X$ closed of codimension $1$, $Z \subset X$ of codimension $\ge d$ such that $W = Y \cap Z$ is of codimension $\ge d$ in $Y$.
Then \[ i^*: H^d_Z(X, M) \to H^d_W(Y, M) \] only depends on $M_{-d}$.
\end{lemma}
\begin{proof}
By a continuity argument, we may assume that $X$ is smooth over $k$.

Using Remark \ref{rmk:RS-basic}, we may shrink $X$ around a generic point of $W$.
We may thus assume that $W$ is smooth over $k$, and connected.
Pullback along the smooth map $X \times \A^1 \to X$ yields an understood isomorphism $H^d_Z(X, M) \to H^d_{Z \times \A^1}(X \times \A^1, M)$, functorial in $X$.
It hence suffices to understand pullback along $i \times \A^1$.
Let $w=Spec(F)$ be a generic point of $W \times \A^1$ of codimension $d$ on $Y \times \A^1$.
Then $w$ lies over the generic point of $\A^1$ \cite[Tag 0CC1]{stacks}.
We may thus (using Remark \ref{rmk:RS-basic} again) pass to the generic fiber over $Spec(k(t)) \in \A^1$; essentially we have base changed the entire problem to $k(t)/k$.
Let us denote the base change of $X$ by $X_1$, and so on.
Since $Z$ is geometrically reduced over $k$ \cite[Tag 020I]{stacks}, its base change $Z_1$ is geometrically reduced over $k(t)$ \cite[Tag 0384]{stacks}.
Lemma \ref{lemm:nbd} below supplies us with an étale neighborhood $X_2 \to X_1$ of $w$ and a smooth map $X_2 \to W_1$ such that $Z_2 \to X_2 \to W_1$ is generically smooth and $Y_2 \to W_1$ is smooth.
Let $X_3 = X_2 \times_{W_1} \{w\}$.
Our base changes are illustrated in the following diagram
\begin{equation*}
\begin{CD}
X_3 @>>> X_2 @>>> X_1 = X \otimes k(t) @>>> X \times \A^1 @>>> X \\
@VVV     @VVV \\
w @>>> W_1 = W \otimes k(t).
\end{CD}
\end{equation*}
By construction $X_3 \to X_1$ is a pro-(étale neighborhood) of $w \in X \otimes k(t)$, and so (using again Remark \ref{rmk:RS-basic}) we may replace $X \otimes k(t)$ by $X_3$.

With these preparatory constructions out of the way, we rename $X_3$ to $X$, $Y_3$ to $Y$ and $Z_3$ to $Z$.
We now have a smooth map $X \to Spec(F)$, where $F$ is infinite (since it contains $k(t)$), $Y \to Spec(F)$ is smooth and $Z \to Spec(F)$ is generically smooth.\NB{The main point of all of these contortions is that in positive characteristic, $Z$ need not be generically smooth over $F$ unless we are careful.}
Also $W = \{w\}$ is an $F$-rational point of $X$, $\dim X = d+1$, $\dim Y = d$ and $\dim Z = 1$.
Shrinking $X$ if necessary may assume that $Y \subset X$ is principal, say cut out by $f \in \scr O(X)$, that every component of $Z$ meets $w$, that $Z$ is smooth away from $w$\NB{being a generically smooth curve, its set of singular points is finite, and so $Z_{sing} \setminus w$ is closed in $X$}, and that $X$ is affine.\NB{Each of these conditions except the last one holds in some open neighborhood of $w$, and their non-empty open intersection contains an affine open of $X$ containing $w$.}

Lemma \ref{lemm:technical} below supplies us with $\bar u: Z \to \A^1$ with $\bar u(w) \ne 0$, $\bar u f: Z \to \A^1$ finite and $\bar u$ having no double roots.
Pick $u \in \scr O(X)$ reducing to $\bar u \in \scr O(Z)$.\NB{possible since $X$ affine}
Then $\phi_1 := uf: X \to \A^1$ is finite when restricted to $Z$, satisfies $\phi_1(w) = 0$, and we claim that $\phi_1$ is smooth at all points of $\phi_1^{-1}(0) \cap Z$.
Note that by construction $(u,f)$ have no common root on $Z$ ($w$ being the only root of $f$ on $Z$), and neither do $(u,du)$ ($u$ not having double roots on $Z$) or $(f, df)$ ($Z(f)$ being smooth).
It follows that $d(uf) = udf + fdu$ does not vanish at points $p \in Z$ with $(uf)(p) = 0$, proving the claim.

We may thus apply Lemma \ref{lemm:gabber'} below to obtain $\phi_2, \dots, \phi_{d+1}: X \to \A^1$ such that $\phi = (\phi_1, \dots, \phi_{d+1}): X \to \A^{d+1}$ is étale at all points of $\phi_1^{-1}(0) \cap Z$, has $\phi(w) = 0$, and there exists an open neighborhood $0 \in U \subset \A^{d}_F$ such that $Z_U \xrightarrow{\phi} \A^1_U$ is a closed immersion (in fact $U=U_1 \times \A^{d-1}$).
The non-étale locus of $\phi$ meets $Z$ in finitely many points (namely a closed subset not containing $w$, and hence no component of the curve $Z$), none of which map to $0$ under $\phi_1$.
Shrinking $U$ further, we may thus assume that $Z_U$ is contained in the étale locus $V$ of $\phi$.
Since $Z_U \wequi \phi(Z_U) \to \phi^{-1}(\phi(Z_U)) \cap V$ is a section of a separated unramified morphism it is clopen \cite[Tag 024T]{stacks}, i.e. we have $\phi^{-1}(\phi(Z_U)) \cap V = Z_U \coprod Z'$ with $Z'$ closed in $V$.

Let $D = D(u)$; this is an open neighborhood of $w$ in $X$.
Note that $Z(uf) \cap D = Y \cap D$.
Let \begin{gather*} X' = (\phi^{-1}(\A^1_U) \cap V) \setminus Z', \\ U_0 = U \cap (\{0\} \times \A^{d-1}) \text{ and } Y' = \phi^{-1}(\A^1_{U_0}) \cap X' = Z(uf) \cap X'. \end{gather*}
We have a commutative diagram of schemes
\begin{equation*}
\begin{CD}
X @<i<< Y @<<< Y \cap D \cap X' \\
@AAA    @.       @|     \\
X' @<<< Y' @<<< Y' \cap D \\
@V{\phi}VV @V{\psi}VV \\
\A^1_U @<j<< \A^1_{U_0}.
\end{CD}
\end{equation*}
Here $j$ is the canonical closed immersion, and $\psi$ is the restriction (i.e. base change) of $\phi$.
In particular $\psi$ is étale and $Y'$ is smooth.
By construction, $\phi$ is an étale neighborhood of $Z_U$, and $Z_U \xrightarrow{\phi} \A^1_U \to U$ is finite.
There is an induced commutative diagram
\begin{equation*}
\begin{CD}
H^d_Z(X, M) @>{i^*}>> H^d_W(Y, M) @>>{\wequi}> H^d_W(Y \cap D \cap X', M) \\
@VoVV                   @.               @|                \\
H^d_{Z_U}(X', M) @>>> H^d_{Z_U \cap Y'}(Y', M) @>o>> H^d_W(Y' \cap D, M) \\
@A{\phi^*}A{\wequi}A               @A{\psi^*}A{\wequi}A \\
H^d_{Z_U}(\A^1_U, M) @>{j^*}>> H^d_{Z_U \cap Z(uf)}(\A^1_{U_0}, M).
\end{CD}
\end{equation*}
We need to understand the top left hand horizontal map.
All the labelled isomorphisms are pullback along étale maps, and isomorphisms by excision.
The two maps labelled $o$ are also pullback along étale morphisms (in fact open immersions), and hence understood.

We have thus reduced to understanding $j^*$; we rename $Z_U$ to $Z$ and $Z_U \cap Z(uf)$ to $V$.
Since $Z$ is finite over $U$, it remains closed in $\P^1_U$, and hence by a further excision argument it suffices to understand \[ \bar j^*: H^d_Z(\P^1_U, M) \to H^d_V(\P^1_{U_0}, M). \]
We have $V = \{w, z_1, \dots, z_r\}$, and each $z_i$ is a smooth point of $Z$ (since $w$ is the only singular point of $Z$).
Moreover $z_i$ is a smooth point of $V$, since $z_i$ is a simple root of $u$ (since by construction $u$ has no double roots on $Z$).
By Lemma \ref{lemm:transverse-pullback}, the pullback \[ j_s^*: H^d_Z((\P^1_U)_{z_s}, M) \to H^d_V((\P^1_{U_0})_{z_s}, M) \] only depends on $M_{-d}$.
Thus applying Lemma \ref{lemm:pullback-transfer}\NB{$\eta_1 = w$ is a rational point, so maps birationally to its image}, it suffices to understand \[ k^*: H^{d-1}_{Z'}(U, M_{-1}) \to H^{d-1}_{V'}(U_0, M_{-1}); \] here $Z'$ and $V'$ are the images of $Z$ and $V$ in $U$.
If $d=1$ we are done by Example \ref{ex:induction-start}.
The general case (i.e. $d>1$) now follows by induction (i.e. restart the argument with $(U, U_0, Z')$ in place of $(X,Y,Z)$).
\end{proof}

\begin{lemma} \label{lemm:nbd}
Let $X$ be a smooth scheme over an infinite field $K$, $W \subset X, Y \subset X$ smooth closed subschemes, $W \subset Z \subset X$ with $Z \subset X$ closed and $Z$ geometrically reduced over $K$ (but not necessarily smooth).
Let $w \in W \cap Y$ such that $\dim_w W \le \dim_w Y$.
There exists an étale neighborhood $X' \to X$ of $w$ together with a smooth morphism $X' \to W$ such that $Z \times_X X' \to W$ is generically smooth (that is, the smooth locus is dense in the source) and $Y \times_X X' \to W$ is smooth.
\end{lemma}
\begin{proof}
We modify \cite[Corollary 5.11]{deglise-regular-base}.\NB{In char. 0, this requires no modification.}
Shrinking $X$ around $w$, we may assume that $X$ is affine and there exist $f_1, \dots, f_d \in \scr O(X)$ such that $W = Z(f_1, \dots, f_d)$ and $W$ has codimension everywhere exactly $d$ in $X$.
Let $\{z_1, \dots, z_r\} \subset Z$ be a choice of smooth point in every component of $Z$, which exist because $Z$ is geometrically reduced \cite[Tag 056V]{stacks}.
Let $\dim X = d+n$.
We claim that there exist $g_1, \dots, g_n \in \scr O(X)$ such that $dg_1, \dots, dg_n$ are linearly independent (over the respective residue fields) in $\Omega_w W, \Omega_w Y$ and $\Omega_{z_i} Z$ for every $i$; we shall prove this at the end.
It follows that $df_1, \dots, df_d, dg_1, \dots, dg_n$ are linearly independent in $\Omega_w X$\NB{Using $0 \to C_{W/X} \to \Omega_X \to \Omega_W \to 0$.}.
Consider the map $F = (f_1, \dots, f_d, g_1, \dots, g_n): X \to \A^{d+n}$.
Let $p: \A^{d+n} \to \A^n$ be the projection to the last $n$ coordinates.
By \cite[17.11.1]{EGAIV}, $F$ is smooth at $w$, $pF|_W$ is smooth at $w$, $pF|_Y$ is smooth at $w$, and $pF|_Z$ is smooth at $z_i$.
In particular $pF|_Z$ is generically smooth.
Shrinking $X$ further around $w$, we may assume that $F, pF|_W$ and $pF|_Y$ are smooth (whence the former two are étale), and of course $pF|_Z$ remains generically smooth.
Applying the construction of \cite[\S5.5, \S5.9]{deglise-regular-base}, we obtain a commutative diagram as follows
\begin{equation*}
\begin{CD}
\Omega @>j>> P @>>> \A^d_W @>>> W \\
@.           @VqVV    @VVV      @V{pF|_W}VV \\
         @. X  @>{F}>> \A^{d+n} @>{p}>> \A^n.
\end{CD}
\end{equation*}
Here both squares are cartesian by definition, $j$ is an open immersion and $qj$ is an étale neighborhood of $W$ (by construction of $\Omega$).
Since $p, F$ and $j$ are smooth so is $\Omega \to W$.
By construction $Z \to X \to \A^n$ is generically smooth, thus so $Z \times_X P \to W$, and hence also $Z \times_X \Omega \to W$.
Similarly $Y \to X \to \A^n$ is smooth and hence so is $Y \times_X \Omega \to W$.
Setting $X' = \Omega$, the result follows.

It remains to prove the claim.
Embed $X$ into $\A^N$, and let $g_1, \dots, g_n$ be linear projections.
By Lemma \ref{lemm:general-basis} below applied to $(\A^N)^* \otimes_K K(w) \subset (\Omega \A^N) \otimes_K K(w) \rightarrow \Omega_w W$, $dg_1, \dots, dg_n$ will be linearly independent in $\Omega_w W$ for general $g_j$, and similarly for $\Omega_w Y$ and $\Omega_{z_i} Z$.
The result follows.
\end{proof}

For completeness, we include a proof of the following elementary fact.
\begin{lemma} \label{lemm:general-basis}
Let $K'/K$ be a finite field extension.
Let $V$ be a finite dimensional $K$-vector space, $V'$ a $K'$-vector space of dimension $\ge n$, and $V_{K'} \to V'$ a surjection.
Write $\A(V)$ for the associated variety over $K$ (isomorphic to $\A^{\dim V}_K$).
There is a non-empty open subset $U$ of $\A(V)^n$ such that any $(v_1, \dots, v_n) \in U(K)$ have $K'$-linearly independent images in $V'$.
In particular, if $F$ is infinite, then $n$ general elements of $V$ are linearly independent in $V'$.
\end{lemma}
\begin{proof}
Replacing $V'$ by a quotient of dimension $n$, we may assume that $\dim_{K'} V' = n$.
There is a map $D: \A(V')^n \to \A_{K'}^1$ such that $n$ elements of $V'$ are linearly independent if and only if their image under $D$ is non-zero (pick a basis of $V'$ and let $D$ be the determinant).
By adjunction the composite \[ \A(V)^n_{K'} \to \A(V')^n \xrightarrow{D} \A^1_{K'} \] defines a map $D': \A(V)^n \to R(\A^1_{K'})$, where $R$ denotes the Weil restriction along $K'/K$ (see e.g. \cite[\S7.6]{NeronModels}).
\NB{We have $R(\A^1_{K'}) \wequi \A^{[K':K]}_K$} By construction $n$ elements of $V$ have image in $V'$ linearly independent over $K'$ if and only if their image under $D'$ is non-zero.
Since $D$ is not the zero map neither is $D'$, and hence $U = D'^{-1}(R(\A^1_{K'}) \setminus 0)$ is the desired non-empty open subset.

The last statement follows since non-empty open subsets of affine space over an infinite field have rational points.
\end{proof}

\begin{lemma} \label{lemm:technical}
Let $Z$ be an affine curve over the field $F$, smooth away from a rational point $w$, and let $f: Z \to \A^1$ be nowhere constant.
There exists $u: Z \to \A^1$ such that $u(w) \ne 0$ and $fu: Z \to \A^1$ is finite.
If $F$ is infinite, it can be arranged that $u$ has no double zeros.
\end{lemma}
\begin{proof}
Let $\bar Z$ be a compactification of $Z$ which is smooth away from $w$, and $\bar Z \setminus Z = \{z_1, \dots, z_r\}$.
Since $Z$ is smooth at infinity, any map $Z \to \A^1$ extends to $\bar Z \to \P^1$.
It suffices to find $u_i: Z \to \A^1$ such that $\ord_{z_i}(u_i) + \ord_{z_i}(f) < 0$, for every $i$.
Then $u = \sum_i u_i^{e_i}$ for suitably big $e_i$ will satisfy the same condition, but for all $i$ at once.\NB{details?}
Now $uf: \bar Z \to \P^1$ is proper\NB{thus finite...} and $\bar Z \setminus Z \supset (uf)^{-1}(\infty) \supset \{z_1, \dots, z_r\}$, which implies that $uf: Z \to \A^1$ is finite (being proper and affine).
If $u(w) = 0$ then replace $u$ by $u+1$; the first claim follows.
Replacing $u$ by $u^n + g$ for suitable $g$ and $n$ large, we may assume that $du$ has only finitely many zeros.
Then $u+c$ for general $c$ has no double zeros (away from $w$) and satisfies $u(w) \ne 0$, so that if $F$ is infinite we may arrange the second claim.

Let $\tilde {\bar Z}$ be the normalization of $\bar Z$, and $\tilde Z \subset \tilde {\bar Z}$ the open subset over $Z$, i.e. the normalization of $Z$.
Note that $\tilde{\bar Z} \to \bar Z$ is an isomorphism near $z_i$.
By Riemann--Roch, we can find $v: \tilde{\bar Z} \to \P^1$ with an arbitrarily large pole at $z_i$, no poles away from $z_i$, so in particular no poles on $\tilde Z$.
Since $\tilde Z \to Z$ is integral, there exists an equation $v^n + a_1 v^{n-1} + \dots + a_n = 0$, with $a_i: Z \to \A^1$.
It follows that (at least) one of the $a_i$ must have a pole at $z_i$ (at least) as large as $v$.
This concludes the proof.
\end{proof}

We used above the following variant of the second part of Gabber's Lemma \cite[Lemma 3.1(b)]{gabber1994gersten}; our proof is heavily inspired by Gabber's.
\begin{lemma} \label{lemm:gabber'}
Let $F$ be an infinite field, $X$ smooth and affine of dimension $d$ over $F$, $Z \subset X$ closed, $w \in Z$ a rational point, $e<d$, $\phi' = (\phi_1, \dots, \phi_e): X \to \A^e$ such that $\phi'(w) = 0$, $\phi'|_Z: Z \to \A^e$ is finite and $\phi'$ is smooth at all points of $\phi'^{-1}(0) \cap Z$.

Then there exist $\phi_{e+1}, \dots, \phi_d: X \to \A^1$ such that $\phi = (\phi_1, \dots, \phi_d): X \to \A^d$ is étale at all points of $\phi'^{-1}(0) \cap Z$, $\phi(w) = 0$ and there exists an open neighborhood $0 \in W \subset \A^{e}$ such that $Z_W \to \A^{d-e}_W$ is a closed immersion.
\end{lemma}
Note that the new map $\phi$ is also finite when restricted to $Z$.
\begin{proof}
Let $X \hookrightarrow \A^N$ with $w$ mapped to $0$.  
We claim that general linear projections $\phi_{e+1}, \dots, \phi_d$ have the desired properties.
They vanish on $w$ by definition.

In order for $\phi$ to be smooth at some point $x \in \phi'^{-1}(0) \cap Z$, we need only ensure that $d\phi_1, \dots, d\phi_d \in \Omega_x X$ are linearly independent \cite[17.11.1]{EGAIV}.
Since $\phi'$ is smooth at $x$, the $d\phi_1, \dots, d\phi_e$ are linearly independent at $x$, and then the other $d\phi_i$ are linearly independent, for general $\phi_i$; this follows from Lemma \ref{lemm:general-basis} applied to $V' = \Omega_x X/\lra{d\phi_1, \dots, d\phi_e}$. 
Since $\phi'^{-1}(0) \cap Z$ is a finite set of points, the étaleness claim holds for general $\phi_i$\NB{étale = smooth + relative dimension 0}.

It remains to prove the claim about the closed immersion.
Note that by Nakayama's lemma, if $f: X \to Y$ is a morphism of affine $S$-schemes with $X$ finite over $S$, $S$ noetherian, and there exists $s \in S$ such that $f_s: X_s \to Y_s$ is a closed immersion, then there exists an open neighborhood $U$ of $s$ such that $f_U: X_U \to Y_U$ is a closed immersion.\NB{ref? Let $S=Spec(R), X=Spec(B), Y=Spec(A)$. Since $A/m \to B/m$ is surjective, $A_m \to B_m$ hits generators of $B_m$ as an $R_m$-module by Nakayama, whence the result.}
Let $\psi: Z \to \A^d$ be the restriction of $\phi$, which we view as a morphism over $S=\A^e$ via $\phi'$ (and the projection $\A^d \to \A^e$ to the first $e$ coordinates).
It is thus enough to show that $\psi_0: \phi'^{-1}(0) \cap Z \to \A^{d-e}$ is a closed immersion (for general $\phi_i$)
Since $\phi$ is étale at all points of $\phi'^{-1}(0) \cap Z$ (for general $\phi_i$) and $Z \to X$ is a closed immersion, $\psi$ is unramified at all points above $0$ and so $\psi_0$ is unramified (for general $\phi_i$).
Since $\phi: Z \to \A^d$ is finite so is $\psi$; in fact $\phi'^{-1}(0) \cap Z$ is finite over $F$.
By \cite[Tags 04XV and 01S4]{stacks}, a morphism is a closed immersion if and only if it is proper, unramified and radicial; we already know that $\psi_0$ is finite (hence proper) and unramified.
Being radicial is fpqc local on the target \cite[Tag 02KW]{stacks}, so may be checked after geometric base change.
In other words (using that $\phi'^{-1}(0) \cap Z$ is finite over $F$) we need the $\phi_i$ to separate a finite number of specified geometric points.
This clearly holds for general $\phi_i$.\NB{Since $e<d$!}
\end{proof}

\begin{proof}[Proof of Theorem \ref{thm:formula}]
If the theorem holds for composable maps $f$ and $g$, then it holds for $fg$.
Given $f: Y \to X$, we factor it as \[ Y \xrightarrow{i} \Gamma_f \xrightarrow{p} X; \] here $i: Y \to \Gamma_f$ is the graph of $f$.
Then $i$ is a regular immersion and $p$ is smooth.
It follows that $p^{-1}(Z) \subset \Gamma_f$ has codimension $\ge d$ (see e.g. \cite[Corollary 6.1.4]{EGAIV}).
Hence it suffices to prove the result for $i, p$ separately; i.e. we may assume that $f$ is either a regular immersion or a smooth morphism.
The case of smooth maps was already explained in Remark \ref{rmk:RS-basic}, so assume that $f: Y \hookrightarrow X$ is a regular immersion.
As usual we may localize in a generic point of $Z \cap Y$ of codimension $d$ on $Y$; hence we may assume that $Z \cap Y = \{z\}$ is a closed point, $X$ is local and $\dim Y = d$.
It follows (e.g. from \cite[Tag 00NQ]{stacks}) that there exists a sequence of codimension $1$ embeddings of essentially smooth schemes \[ Y = Y_d \hookrightarrow Y_{d+1} \hookrightarrow \dots \hookrightarrow Y_n = X. \]
Let $Z_i = Z \cap Y_i$.
Then $\dim Z_i \ge \dim Z_{i+1}-1$ \cite[Tag 00KW]{stacks} but $\dim Z_n = n-d = \dim Z_d + n-d$ so that $\dim Z_i = i-d$ for all $i$.
It follows that we may prove the result for each $Y_i \hookrightarrow Y_{i+1}$ separately; this is Lemma \ref{lemm:key}.
\end{proof}

\section{Applications} \label{sec:app}
After introducing some notation in \S\ref{subsec:app-notation}, we identify the heart of $\SHS(k)(d)$ (for $d \ge 3$) in \S\ref{subsec:app-heart}.
This establishes \cite[Conjecture 6.10]{bachmann-moving}.
Finally in \S\ref{subsec:app-res} we study convergence of the resolution of an $S^1$-spectrum by infinite $\P^1$-loop spectra arising from the adjunction $\SHS(k) \adj \SH(k)$ and deduce some conservativity results.

We also assume that $k$ is perfect; we will restate this assumption with the most important results only.

\subsection{Notation and hypotheses} \label{subsec:app-notation}
We write $\SHS(k)$ for the category of motivic $S^1$-spectra \cite[\S4]{morel-trieste} (i.e. the motivic localization of the category of spectral presheaves on $\Sm_k$), and $\SH(k) = \SHS(k)[\Gmp{-1}]$ for the category of motivic spectra \cite[\S5]{morel-trieste}.
We use the following notation for the stabilization functors \[ \Sm_{k*} \xrightarrow{\Sigma^\infty_{S^1}} \SHS(k) \xrightarrow{\sigma^\infty} \SH(k)^\eff \subset \SH(k), \] and we denote by $\omega^\infty$ the right adjoint of $\sigma^\infty$.
Here $\SH(k)^\eff$ is the localizing subcategory generated by the image of $\sigma^\infty$.

There are localizing subcategories \[ \SHS(k) \supset \SHS(k)(1) \supset \dots \supset \SHS(k)(d) \supset \dots; \] here $\SHS(k)(d)$ is generated by $\Sigma^\infty_{S^1} X_+ \wedge \Gmp{d}$ for $X \in \Sm_k$.
The inclusion $\SHS(k)(d) \subset \SHS(k)$ has a right adjoint which we denote by $f_d$.\footnote{We abuse notation somewhat and view this as a functor $\SHS(k) \to \SHS(k)$.}
There are canonical cofiber sequences $f_{d+1} \to f_d \to s_d$ defining the functors $s_d$.
There is a similar filtration of $\SH(k)^\eff$, given by $\SH(k)^\eff(d) := \SH(k)^\eff \wedge \Gmp{d}$, and the right adjoints (respectively cofibers) are again denoted by $f_d$ (respectively $s_d$).
See \cite{levine2008homotopy,voevodsky-slice-filtration} or \cite[\S6.1]{bachmann-moving} for more details on these functors.

Recall that $\SHS(k)$ has a $t$-structure with non-negative part generated by $\Sigma^\infty_{S^1} X_+$ for $X \in \Sm_k$; its heart canonically identifies with $\HI(k)$ \cite[Lemma 4.3.7(2)]{morel-trieste}.
We denote by $E_{\ge 0}, E_{\le 0}$ and $\ul{\pi}_0 E$ the truncations and homotopy sheaves, respectively.
The categories $\SHS(k)(d)$, $\SH(k)$, $\SH(k)^\eff(d)$ have related $t$-structures, with non-negative parts generated by $X_+ \wedge \Gmp{d}$.

Recall from \cite[\S5.1]{bachmann-moving} the notion of a presheaf with $\A^1$-transfers.
This is just a presheaf $F$ on $\Sm_k$ together with for every finitely generated field $K/k$ a $GW(K)$-module structure on $F(K)$, and for every finite monogeneic extension $K(x)/K$ a transfer $\tau_x: F(K(x)) \to F(K)$.
The category $\SH(k)^{\eff\heart}$ embeds fully faithfully into the category of presheaves with $\A^1$-transfers \cite[Corollary 5.17]{bachmann-moving} (morphisms in this category are given by morphisms of presheaves compatible with the $GW$-module structures and transfers).
Given a presheaf with $\A^1$-transfers $M$, we say that \emph{the transfers extend to framed transfers} if $M$ is in the essential image of this embedding.
Recall also that Morel has shown that if $M \in \HI(k)$ and $d>0$, then $M_{-d}$ canonically extends to a presheaf with $\A^1$-transfers (see \S\ref{subsec:transfer} or \cite[Example 5.2]{bachmann-moving}).
\begin{definition}
Let $k$ be a perfect field and $d>0$.
\begin{enumerate}
\item Let $M \in \HI(k)$.
  We shall say that hypothesis $T_d(M)$ holds if the canonical $\A^1$-transfers on $M_{-d}$ extend to framed transfers.
  We shall say that hypothesis $T_d(k)$ holds if $T_d(M)$ holds for all $M \in \HI(k)$.
\item We shall say that hypothesis $S_d(k)$ holds if for any $E \in \SHS(k)$ and $i \in \Z$ the spectrum $f_d \ul{\pi}_i s_d E$ is in the essential image of $\omega^\infty: \SH(k) \to \SHS(k)$.
\end{enumerate}
\end{definition}

\begin{remark}
If $k$ is perfect then $T_d(k)$ holds for any $d \ge 3$, and if $char(k) = 0$ then $T_2(k)$ also holds \cite[Theorem 5.19]{bachmann-moving}.
\end{remark}
\begin{remark} \label{rmk:speculation}
We speculate that $T_1(k)$ holds for any perfect field.
\end{remark}
\begin{remark} \label{rmk:Td-bc}
If $f: Spec(l) \to Spec(k)$ is an algebraic extension (automatically separable) and $M \in \HI(k)$ such that $T_d(M)$ holds, then also $T_d(f^*M)$ holds.
This is obvious for $f$ finite, and the general case follows by continuity and essentially smooth base change.
\end{remark}

\begin{theorem}[Levine \cite{levine2010slices}]
Let $char(k) = 0$.
Then $S_d(k)$ holds for any $d>0$.
\end{theorem}
\begin{proof}
This is essentially \cite[Theorem 2]{levine2010slices}; we just have to show that \[ s_{p,n} E \wequi f_n \Sigma^{p+n} \ul{\pi}_{p+n} s_n E. \]
By definition \cite[main construction (9.2)]{levine2010slices}, $s_{p,n} E(X)$ is the realization of (a rectification of) the simplicial object $\Sigma^p\pi_p (s_n E)^{(n)}(X, \bullet)$; here $(s_n E)^{(n)}(X, \bullet)$ is the homotopy coniveau tower model of $f_n s_n E \wequi s_n E$, and $\pi_p$ just means taking the $p$-th Eilenberg-MacLane spectrum of the (levelwise) ordinary spectrum $(s_n E)^{(n)}(X, \bullet)$.

The map $(s_n E)_{\ge n+p} \to s_n E$ of spectral sheaves induces a map \[ \alpha_p: ((s_n E)_{\ge p+n})^{(n)}(\ph, \bullet) \to (s_n E)^{(n)}(\ph, \bullet) \] of simplicial spectral presheaves.
I claim that $\alpha_p$ induces an isomorphism on $\pi_i$ for $i \ge p$, and that the source has $\pi_i = 0$ for $i < p$.
This yields an equivalence \[ ((s_n E)_{\ge p+n})^{(n)}(\ph, \bullet) \wequi \tau_{\ge p}(s_n E)^{(n)}(\ph, \bullet), \] where $\tau_{\ge p}$ just means levelwise truncation of the simplicial presheaf of spectra.
Taking cofibers we obtain \[ (\Sigma^{p+n} \ul{\pi}_{p+n} s_n E)^{(n)}(\ph, \bullet) \wequi \Sigma^p \pi_p(s_n E)^{(n)}(\ph, \bullet), \] which is what we set out to prove (using \cite[Theorem 7.1.1]{levine2008homotopy}).

It is hence enough to prove the claim.
Thus let $X \in \Sm_k$ and $W \subset \A^m_X$ have codimension $\ge n$.
The definition of the homotopy coniveau tower (recalled e.g. in \cite[\S9]{levine2010slices}) implies that it is enough to show that \[ H^{-i}_W(\A^m_X, s_n E) \wequi H^{-i}_W(\A^m_X, (s_n E)_{\ge n+p}) \text{ for } i\ge p \] and \[ H^{-i}_W(\A^m_X, (s_n E)_{\ge n+p}) = 0 \text{ for } i < p. \]
Considering the (strongly convergent) descent spectral sequence (for $F \in \SHS(k)$)\NB{ref?} \[ H^p_W(\A^m_X, \ul{\pi}_q F) \Rightarrow H^{p-q}_W(\A^m_X, F), \] for this it suffices to show that for $j \in \Z$ we have \[ H^i_W(\A^m_X, \ul{\pi}_j s_n E) = 0 \text{ for } i \ne n. \]
We can compute this cohomology group using the Rost--Schmid resolution; since $W$ has codimension $\ge n$ the vanishing follows from the observation that $\ul{\pi}_j(s_n E)_{-i} = 0$ for $i > n$, which holds since $\Omega_{\Gm}^{i} s_n E \wequi 0$ (for $i>n$), by definition.
\end{proof}

\subsection{The heart of $\SHS(k)(d)$} \label{subsec:app-heart}
Consider the adjunction \[ \sigma^\infty: \SHS(k) \adj \SH(k): \omega^\infty. \]
Then $\sigma^\infty(\SHS(k)(d)) \subset \SH(k)^\eff(d)$ and $\sigma^\infty(\SHS(k)(d)_{\ge 0}) \subset \SH(k)^\eff(d)_{\ge 0}$, for any $d \ge 0$.
Moreover it follows from \cite[Lemmas 6.1(2) and 6.2(1,2)]{bachmann-moving} that $\omega^\infty(\SH(k)^\eff(d)) \subset \SHS(k)(d)$ and $\omega^\infty(\SH(k)^\eff(d)_{\ge 0}) \subset \SHS(k)(d)_{\ge 0}$.
This implies that there is an induced adjunction \[ \pi_0^d \sigma^\infty: \SHS(k)(d)^\heart \adj \SH(k)^\eff(d)^\heart: \omega^\infty, \] where $\pi_0^d$ denotes the truncation functor in the $t$-structure on $\SH(k)^\eff(d)$.
\begin{theorem} \label{thm:hearts}
Let $k$ be a perfect field such that $T_d(k)$ holds.
Then the functor $\omega^{\infty}: \SH(k)^\eff(d)^\heart \to \SHS(k)(d)^\heart$ is an equivalence of categories.
\end{theorem}
This establishes \cite[Conjecture 6.10]{bachmann-moving} (for $n=d$).
\begin{proof}
The functor is fully faithful by \cite[Theorem 6.9]{bachmann-moving}; it hence suffices to prove essential surjectivity.
We shall prove the following more precise statement: if $M \in \HI(k)$ and $T_d(M)$ holds, then $M$ is in the essential image of $\omega^\infty$.

We first prove this assuming that $k$ is infinite.
We have $\ul{\pi}_i(M)_{-d} = 0$ for $i \ne 0$ \cite[Lemma 6.2(3)]{bachmann-moving} and hence the canonical map $M \to f_d \ul{\pi}_0 M$ is an equivalence (indeed it induces an equivalence on $\ul{\pi}_i(\ph)_{-d}$ for every $i$, and this detects equivalence in $\SHS(k)(d)$ by \cite[Lemma 6.1(1)]{bachmann-moving}).
By assumption, the $\A^1$-transfers on $\ul{\pi}_0(M)_{-d}$ extend to framed transfers; there hence exists $\tilde M \in \SH(k)^{\eff\heart}$ such that $\omega^\infty(\tilde M)_{-d} \wequi \ul{\pi}_0(M)_{-d}$ as presheaves with $\A^1$-transfers.\NB{i.e. framed sheaves extend to homotopy modules}
By Lemma \ref{lemm:main} below (this is where we use the assumption that $k$ is infinite), this implies that $f_d \omega^\infty(\tilde M) \wequi f_d \ul{\pi}_0(M)$.
It follows from \cite[Theorem 9.0.3]{levine2008homotopy} that $f_d$ commutes with $\omega^\infty$; we thus find that \[ M \wequi f_d \ul{\pi}_0 M \wequi f_d \omega^\infty(\tilde M) \wequi \omega^\infty f_d \tilde M. \]
The claim is thus proved for $k$ infinite.

Now let $k$ be finite and $M \in \HI(k)$ such that $T_d(M)$ holds.
Since $\omega^\infty$ is fully faithful, $M$ is in the essential image of $\omega^\infty$ if and only if the canonical map $M \to \omega^\infty \sigma^{\infty\heart} M$ is an isomorphism.
The functors $\omega^\infty$ and $\sigma^{\infty\heart}$ commute with essentially smooth base change.
Let $f: Spec(l) \to Spec(k)$ be an infinite algebraic $p$-extension of $k$, for some prime $p$.
Using Lemma \ref{lemm:conservative} below we reduce to proving that $f^*M$ is in the essential image of $\omega^\infty$.
By Remark \ref{rmk:Td-bc}, $T_d(f^*M)$ holds, and thus we are reduced to what was already established.

This concludes the proof.
\end{proof}

\begin{lemma} \label{lemm:main}
Let $k$ be an infinite perfect field, $M, N \in \HI(k)$ and $d>0$.
Suppose that $T_d(M)$ holds.
Any isomorphism $M_{-d} \wequi N_{-d}$ respecting the $\A^1$-transfers yields an equivalence $f_d M \wequi f_d N$.
\end{lemma}
\begin{proof}
We have $f_d M \wequi M^{(d)}$ \cite[Theorem 7.1.1]{levine2008homotopy} (this is where we use the assumption that $k$ is infinite).
As explained in \cite[Remark 4.17]{bachmann-moving}, the (truncated) BLRS complex of $M$ provides a model of $M^{(d)}$ which only depends on $M_{-d}$ as a presheaf of $GW$-modules together with the maps $f^*: H^d_Z(X, M) \to H^d_{f^{-1}(Z)}(Y, M)$ for $f: Y \to X \in \Sm_k$, $Z \subset X$ closed of codimension $\ge d$ such that $f^{-1}(Z)$ also has codimension $\ge d$.
(In order to apply this remark, we need to know that $M_{-d}$ has framed transfers (see e.g. \cite[Proposition 4.14]{bachmann-moving}); this is the only reason for assuming $T_d(k)$.)
Theorem \ref{thm:formula} shows that $f^*$ only depends on $M_{-d}$ as a presheaf with $\A^1$-transfers.
The result follows.
\end{proof}

\begin{lemma} \label{lemm:conservative}
Let $k$ be a perfect field and $k_p/k$ (respectively $k_q/k$) a separable algebraic $p$-extension (respectively $q$-extension), for primes $p \ne q$.
Let $\alpha: E \to F \in \SHS(k)(d)$ such that $T_d(M)$ holds for any homotopy sheaf $M$ of $E$ or $F$.\NB{probably this results holds for any $d>0$ and all $E$, no assumption on $M$, but proof slightly harder}
If the image of $\alpha$ in $\SHS(k_p)(d) \times \SHS(k_q)(d)$ is an equivalence, then so is $\alpha$.
\end{lemma}
\begin{proof}
It suffices to prove that $\Omega^d_\Gm(\alpha)$ is an equivalence \cite[Lemma 6.1(1)]{bachmann-moving}, i.e. that $\ul{\pi}_i(\alpha)_{-d}$ is an isomorphism for all $i$.
By assumption this is a morphism between sheaves admitting framed transfers, and by \cite[Corollary 5.17]{bachmann-moving} the morphism preserves the transfers.
The result thus follows from \cite[Corollary B.2.5]{EHKSY} (using that essentially smooth base change commutes with $\Omega^d_\Gm$).
\end{proof}

The following is our degree zero $\Gm$-Freudenthal theorem.
\begin{corollary} \label{cor:freudenthal}
Suppose that $T_d(k)$ holds.
\begin{enumerate}
\item Let $E \in \SHS(k)_{\ge 0} \cap \SHS(k)(d)$.
  Then \[ \ul{\pi}_0(E) \wequi \ul{\pi}_0(\omega^\infty \sigma^\infty E). \]
\item Let $E \in \SHS(k)_{\ge 0}$.
  Then \[ \ul{\pi}_0(\omega^\infty \sigma^\infty E) \wequi \ul{\pi}_0(\Gmp{d} \wedge E)_{-d}. \]
\end{enumerate}

In particular this holds for $d \ge 3$ if $k$ is perfect, and for $d \ge 2$ if $char(k) = 0$.
\end{corollary}
\begin{proof}
(1) We have $E \in \SHS(k)(d)_{\ge 0}$ \cite[Lemma 6.2(3)]{bachmann-moving}.
Write $\pi_0^d$ for the homotopy object in the $t$-structure on $\SHS(k)(d)$.
Then since $\sigma^\infty, \omega^\infty$ are both right-$t$-exact \cite[Lemma 6.2(1,2)]{bachmann-moving} we learn from Theorem \ref{thm:hearts} that $(*)$ $\pi_0^d E \wequi \pi_0^d \omega^\infty \sigma^\infty E$.
Since $\SHS(k)(d)_{\ge 0} \subset \SHS(k)_{\ge 0}$ (by construction), we have $\ul{\pi}_0 \pi_0^d \wequi \ul{\pi}_0$ (when applied to objects in $\SHS(k)(d)_{\ge 0}$), and hence the result follows by applying $\ul{\pi}_0$ to $(*)$.

(2) We have \begin{align*} \ul{\pi}_0(\omega^\infty \sigma^\infty E) &\wequi \ul{\pi}_0(\omega^\infty \Omega_{\Gm}^d \Sigma_\Gm^d \sigma^\infty E) \\ &\wequi \ul{\pi}_0(\Omega_{\Gm}^d \omega^\infty \sigma^\infty \Sigma_\Gm^d E) \\ &\wequi \ul{\pi}_0(\omega^\infty \sigma^\infty \Gmp{d} \wedge E)_{-d} \\ &\stackrel{(1)}{\wequi} \ul{\pi}_0(\Gmp{d} \wedge E)_{-d}. \end{align*}
\end{proof}

\subsection{Canonical resolutions} \label{subsec:app-res}
In this section we will freely use the language of $\infty$-categories as set out in \cite{HTT,HA}.

Given any adjunction $F: \scr C \adj \scr D: G$ of $\infty$-categories, there is a monad structure on $GF$ \cite[Proposition 4.7.4.3]{HA}.
Hence for $E \in \scr C$ there is a canonical ``triple resolution'' $E \to E^\bullet$, where $E^\bullet$ denotes a cosimplicial object with $E^n = (GF)^{\circ (n+1)}(E)$.
In more detail, by definition of a monad, $GF$ promotes to an $\scr E_1$-algebra in $\Fun(\scr C, \scr C)$ under the composition monoidal structure, and then \cite[Construction 2.7]{mathew2017nilpotence} yields an augmented cosimplical object in $\Fun(\scr C, \scr C)$; the triple resolution is obtained by applying this cosimplicial endofunctor to $E$.
This also makes it clear that the triple resolution is functorial in $E$.

Applying this to the stabilization adjunction $\sigma^\infty: \SHS(k) \adj \SH(k): \omega^\infty$, we obtain the \emph{canonical resolution} \[ E \to E^\wedge := \lim_{\Delta}[(\omega^\infty \sigma^\infty)^{\circ (\bullet + 1)} E], \] functorially in $E$.
\begin{definition}
We shall say that \emph{the canonical resolution converges for $E$} if the above morphism is an equivalence.
\end{definition}

\begin{example} \label{ex:conv1}
The canonical resolution converges if $E$ is in the essential image of $\omega^\infty$, since then the cosimplicial object is split.\todo{ref?}
\end{example}

\begin{example} \label{ex:conv2}
Given a cofiber sequence $E_1 \to E_2 \to E_3$, if the canonical resolution converges for any two of the three terms, then it converges for the third.
This holds since all the functors involved are stable.\todo{details?}
\end{example}

The following result clearly holds in much greater generality, but for simplicity we state it in our restricted context.
\begin{lemma} \label{lemm:convergence}
Let $E \in \SHS(k)$ and suppose given a tower \[ \dots \to E_2 \to E_1 \to E_0 := E \] and a sequence $n_i \in \Z$ such that (i) $\lim_i n_i = \infty$, (ii) $E_i \in \SHS(k)_{\ge n_i}$ (for all $i$) and (iii) the canonical resolution converges for $cof(E_{i+1} \to E_i)$ (for all $i$).

Then the canonical resolution converges for $E$.
\end{lemma}
\begin{proof}
Define an endofunctor $F$ of $\SHS(k)$ by $F(X) = fib(X \to X^{\wedge})$.
By construction this is a stable functor such that $F(X) \wequi 0$ if and only if the canonical resolution converges for $X$.
Since $F(cof(E_{i+1} \to E_i)) \wequi 0$ for all $i$, we find that the tower \[ \dots \to F(E_1) \to F(E_0) = F(E) \] is constant.
In order to prove that $F(E) = 0$ it thus suffices to show that $\lim_i F(E_i) \wequi 0$.
Commuting the limits\NB{And using that constant coaugmented cosimplicial diagrams are limits.}, we find that \[ \lim_i F(E_i) \wequi \lim_{i \ge 0, n \in \Delta} fib(E_i \to (\omega^\infty \sigma^\infty)^{\circ (n+1)} E_i) \wequi \lim_{n \in \Delta} \lim_i fib(E_i \to (\omega^\infty \sigma^\infty)^{\circ (n+1)} E_i); \] it thus suffices to show that the inner limit over $i$ vanishes.
Since $\sigma^\infty, \omega^\infty$ are both right-$t$-exact \cite[Lemma 6.2(1,2)]{bachmann-moving}, we have $fib(\dots) \in \SHS(k)_{\ge n_i - 1}$, and hence it is enough to show that if $X_i$ is a sequence of spectra with $X_i \in \SHS(k)_{\ge n_i}$ then $\lim_i X_i \wequi 0$.
Since $\SHS(k)$ is generated as a localizing subcategory by objects of the form $\Sigma^\infty_{S^1} U_+$ for $U \in \Sm_k$, considering the Milnor exact sequence \cite[Proposition VI.2.15]{goerss2009simplicial} it suffices to show: if $n > \dim U$ then $[\Sigma^\infty_{S^1} U_+, \SHS(k)_{\ge n}] = 0$.
This follows from the descent spectral sequence.\todo{details?}
\end{proof}

\begin{theorem} \label{thm:convergence}
Let $k$ be a perfect field such that $T_d(k)$ holds.
\begin{enumerate}
\item The canonical resolution converges for all $E \in \SHS(k)_{\ge 0} \cap \SHS(k)(d)$.
\item Suppose in addition that $S_j(k)$ holds for all $1 \le j < d$.
  Then the canonical resolution converges for all $E \in \SHS(k)_{\ge 0} \cap \SHS(k)(1)$.
\end{enumerate}
\end{theorem}
\begin{proof}
(1) Write $\tau_{\ge i}^d$ for the truncation in the $t$-structure on $\SHS(k)(d)$.
We have $E \in \SHS(k)(d)_{\ge 0}$ \cite[Lemma 6.2(3)]{bachmann-moving}.
Apply Lemma \ref{lemm:convergence} with $E_i = \tau_{\ge i}^d E$ and $n_i = i$; assumption (i) is clear, and (ii) holds by \cite[Lemma 6.2(1)]{bachmann-moving}.
For (iii), it suffices by Example \ref{ex:conv1} to show that $\pi_i^d E$ is in the essential image of $\omega^\infty$.
This follows from Theorem \ref{thm:hearts}.

(2) By Example \ref{ex:conv2}, (1) and induction, it suffices to show that the canonical resolution converges for $s_j E$, for $1 \le j < d$.
Since $f_j: \SHS(k) \to \SHS(k)$ is right-$t$-exact \cite[Lemma 6.2]{bachmann-moving}, we have $s_j E \in \SHS(k)_{\ge 0}$.
Apply Lemma \ref{lemm:convergence} with $E_i = f_j[(s_j E)_{\ge i}]$ and $n_i = i$.
Assumption (i) is clear, (ii) holds by right-$t$-exactness of $f_j$, and (iii) holds by definition of $S_j(k)$ and Example \ref{ex:conv1}.
\end{proof}

\begin{corollary} \label{cor:conservative}
\begin{enumerate}
\item Suppose that $T_d(k)$ holds.
  Then \[ \sigma^\infty: \SHS(k)_{\ge 0} \cap \SHS(k)(d) \to \SH(k) \] is conservative.
\item Suppose that additionally $S_j(k)$ holds, for $1 \le j < d$.
  Then \[ \sigma^\infty: \SHS(k)_{\ge 0} \cap \SHS(k)(1) \to \SH(k) \] is conservative.
\end{enumerate}

In particular (1) holds for $d = 3$ if $k$ is perfect, and (2) holds if $char(k) = 0$.
\end{corollary}
\begin{proof}
Clearly convergence of canonical resolutions implies conservativity (which is equivalent to detecting zero objects, by considering cofibers), so this is immediate from Theorem \ref{thm:convergence}.
\end{proof}

\bibliographystyle{alphamod}
\bibliography{bibliography}
\end{document}